\documentclass[11pt, reqno]{amsart}
\usepackage[margin=3.6cm]{geometry}
\usepackage{array,booktabs,tabularx}
\usepackage{graphicx}
\usepackage{enumerate}
\usepackage{color}

\newcommand{\R}{{\mathbb R}}
\newcommand{\h}{h}

\renewcommand{\Re}{\mathrm Re}

\newtheorem{theorem}{Theorem}

\newtheorem{proposition}[theorem]{Proposition}
\newtheorem{lemma}[theorem]{Lemma}

\theoremstyle{definition}

\newtheorem{remark}{Remark}
\newtheorem{claim}{Claim}
\title[Perturbations of propagated Schr\"{o}dinger eigenfunctions]
{On the distribution of perturbations of propagated Schr\"{o}dinger eigenfunctions}

\address[Y. Canzani, D. Jakobson, J. Toth]{Department of Mathematics and
Statistics, McGill University, 805 Sherbrooke Str. West, Montr\'eal
QC H3A 2K6, Ca\-na\-da.\medskip}

\author[Y. Canzani]{Yaiza Canzani}
\email{canzani@math.mcgill.ca}

\author[D. Jakobson]{Dmitry Jakobson}
\email{jakobson@math.mcgill.ca}

\author[J. Toth]{John Toth}
\email{jtoth@math.mcgill.ca}

\thanks{Y.C. was supported by Schulich Fellowship. D.J. and J.T. were supported by NSERC, FQRNT and Dawson Fellowships.\\
\indent Corresponding author: John Toth}
\begin{document}


\maketitle

%\tableofcontents

\begin{abstract}

Let $(M,g_0)$ be a compact Riemmanian manifold of dimension $n$. Let 
$P_0 (\h) := -\h^2\Delta_{g}+V$
be the semiclassical Schr\"{o}dinger operator for $\h \in (0,\h_0]$, and
let $E$ be a regular value of  its principal symbol. % $p_0(x,\xi)=|\xi|^2_{g_0(x)} +V(x)$.
Write $\varphi_\h$ for an  $L^2$-normalized eigenfunction of $P_0(\h)$ with eigenvalue
 $E(\h) \in [E-o(1),E+ o(1)]$.
 We consider a smooth family of metric perturbations $g_u$ of $g_0$
with $u$ in the
product space $B^k(\varepsilon)=(-\varepsilon, \varepsilon)^k  \subset \mathbb R^k$ satisfying the admissibility condition in Definition \ref{admissible}. 
For $P_{u}(\h) := -\h^2 \Delta_{g_u}  +V$
and small $|t|>0$, we define the propagated perturbed eigenfunctions
$$\varphi_{h,t}^{(u)}:=e^{-\frac{i}{\h}t P_u(\h) } \varphi_\h.$$
They appear in the mathematical description of the  Loschmidt echo effect 
in physics. 

Motivated by random wave conjectures in quantum chaos, 
we study the distribution of the real part of the perturbed eigenfunctions regarded as  
random variables
$\Re (\varphi^{(\cdot)}_{h,t}(x)): B^{k}(\varepsilon) \to \mathbb R$ for $x\in M$.
In particular, under an admissibility condition on the metric when $(M,g)$ is chaotic, we compute the $h \to 0^+$ asymptotics of
the variance $\text{Var} [\Re(\varphi^{(\cdot)}_{h,t}(x))] $ and show that the odd moments  vanish
as $h \to 0^+$ as long as $x$ is not on the generalized caustic set where $V(x)=E.$ 
\end{abstract}\ \\

%-----------------------------------------------------------------------------------------------------------------------------------
\section{Introduction}

Let $(M,g_0)$ be a compact Riemmanian manifold of dimension $n$  with Laplace operator
$\Delta_{g_0}=\delta_{g_0}d :C^\infty(M) \to C^\infty (M)$, and let $V \in C^\infty(M)$ denote a smooth potential over $M$.
For  $\h \in(0,  \h_0]$, consider the Schr\"{o}dinger operator
\begin{equation}\label{P_0}
P_0 (\h) := -\h^2\Delta_{g_0}+V,
\end{equation}
and let $E$ be a regular value of  its principal symbol $p_{0}(x,\xi) := |\xi|^2_{g_0(x)} +V(x)$. 
Write $\varphi_\h$ for an  $L^2$-normalized eigenfunction of $P(\h)$ 
 with eigenvalue contained in a shrinking interval centered at $E$;
that is,  $P_0(\h)\varphi_\h =E(\h)\varphi_\h$ and $E(\h) \in [E-o(1),E+ o(1)]$. \\

Consider a smooth family of perturbations $g_u$ of the reference metric $g_0$
with $u$ in the
 product space $B^k(\varepsilon)=(-\varepsilon, \varepsilon)^k \subset \mathbb R^k$.
The number of parameters $k \geq n$ is chosen sufficiently large (but finite) so that the admissibility condition on the perturbation $g_u$ in Definition \ref{admissible} is satisfied.
We introduce the associated perturbed Schr\"{o}dinger operators
	\begin{equation}\label{P_u(h)}
	P_{u}(\h) := -\h^2 \Delta_{g_u}  +V,
	\end{equation}
with principal symbol
	\begin{equation}\label{p_u}
		p_u: T^*M \to T^*M, \quad \quad p_{u}(x,\xi) := |\xi|^2_{g_u}  +V(x).
	\end{equation}
Fix $t \neq 0$ small, independent of $h$, and define
 the perturbed propagated eigenfunctions
\begin{equation}\label{perturbed efxs}
\displaystyle{ \varphi_{h,t}^{(u)}:=e^{-\frac{i}{\h}t P_u(\h) } \varphi_\h} .
\end{equation}
\
The perturbations satisfy $\varphi_{h,t}^{(u)}=\Phi_\h^{(u)}(t)$ where  $\Phi_\h^{(u)}(t)$ 
denotes  the solution at time $t$ of the Schr\"{o}dinger equation
    $$\begin{cases}
	\left(i \h \frac{\partial }{\partial s}- P_u(\h) \right )\Phi_\h^{(u)}(s)=0,\\
	\Phi_\h^{(u)}(0)=\varphi_\h.
    \end{cases}$$

The aim of this paper is to study the  $\h \to 0^+$ asymptotics of the distribution 
of $\varphi_{h,t}^{(u)},$  where the latter are regarded as random variables in
$u \in  B^k(\varepsilon)$.  Specifically, we compute the variance and all odd moments 
in the semiclassical limit $\h\to 0^+$.
To state our results, we need to define an admissibility condition on the metric 
perturbations. Here and throughout the rest of the manuscript we adopt the notation 
$\delta_{u_\alpha}=\partial_{u_\alpha}\big|_{u=0}$. \\

\noindent{\bf Definition 1 (Admissibility condition).} \label{admissible}
		Let $g_u$ with $u \in  B^k(\varepsilon)$ be a $C^{\infty}$ metric perturbation of a 
reference metric $g_0$.
		 We say that $g_u$  is  {\bf admissible} at $x \in M$  if \\	
		 
		\begin{enumerate}[A)]{\leftmargin=1em \itemsep=1em}
		\item \label{cond 1}
               There exists an $n$-tuple of coordinates of 
                $u$, $u'=(u_1, \dots, u_n)$,  for which the Hessian matrices	
		$d_{u'} d_\xi ( p_u(x,\xi))$
		are invertible  for all $u \in B^k(\varepsilon)$ and all $\xi \in T_x^*M$ with
		$(x,\xi) \in p_0^{-1}(E-c\varepsilon, E+c\varepsilon)$
		where the constant $c=c(\varepsilon)>0$ is defined in (\ref{constant}).\\
		
		 \item \label{cond 2}
		 There exists a parameter coordinate $u_\alpha$, a neighborhood $\mathcal W$ of $x$, and 
		 a function $a \in C^{\infty}(\mathcal W, \mathbb R \backslash \{0\})$ 
		   such that
		 $ \delta_{u_{\alpha}} g_u^{^{-1}}(x) =a(x)\, g_0^{^{-1}}(x)$ for  $x \in \mathcal W.$ 
		\end{enumerate}\ \smallskip

\begin{remark}\label{neighborhood U}
If $g_u$ is an admissible at $x$, then there exists a neighborhood $\mathcal U\subset M$ of $x$ on 
which condition \eqref{cond 1} holds.
\end{remark}

We show in Section \ref{section: admissible perturbations} that the admissibility 
condition in Definition 1 is satisfied by a large class of metric perturbations and  
we also give a geometric interpretation of the admissibility condition.
The tangent space at $g_0$ to the space of all Riemannian metrics over $M$ can be 
decomposed into the direct sum of the space of symmetric $2$-tensors with the fixed 
volume form $dvol_{g_0}$, and the space of symmetric $2$-tensors obtained by pointwise 
multiplication of $g_0$.
   We show that the notion of being admissible is intrinsically related to having 
$n = \dim M$ volume preserving directions in which the metric $g_0$ is perturbed 
(this is condition \eqref{cond 1}) and to having one direction in which the metric 
can be conformally perturbed (this is condition \eqref{cond 2}). 

As a model example, 
suppose one wishes to perturb the flat metric $g_0$  on the 2-torus $\mathbb T^2$. 
Let $x_0 \in \mathbb T^2$ and in a neighborhood $\mathcal W$ of $x_0$ consider any 
perturbation $g_u$ with $u=(u_1,u_2,u_3,u'')\in B^k(\varepsilon)$, $k \geq 3$, of the form
$$g_u^{-1}(x)=g_0^{-1}(x)+h_{(u_1,u_2,u_3)}(x) +h_{u''}(x),\qquad x \in \mathcal W,$$ 
where
$$ 
h_{(u_1,u_2,u_3)}(x)=
    u_1 \begin{pmatrix} a_1(x) & b_1(x) \\ b_1(x) & -a_1(x) \end{pmatrix}
   + u_2 \begin{pmatrix} a_2(x) & b_2(x) \\ b_2(x) & -a_2(x) \end{pmatrix}
   + u_3 \begin{pmatrix} a_3(x) & 0 \\ 0 & a_3(x) \end{pmatrix}
$$
and $h_{u''}$ is any symmetric $2$-tensor 
depending on the left-out variables $u'' \in B^{k-3}(\varepsilon)$ and higher 
powers of $u_1,u_2$ and $u_3$.
The perturbation $g_u$ is admissible at $x_0$ provided  $\varepsilon$ is small, 
$a_2, a_2, a_3 \in C^\infty(\mathcal W)$,   $a_3 \neq 0$ in $\mathcal W$, and the vector fields
$(a_1(x_0), b_1(x_0))$ and  $(a_2(x_0),b_2(x_0))$ are linearly independent.  
We remark that this admissibility condition is satisfied on open subsets in the space of all $C^{\infty}$
metric perturbations on $\mathbb T^2$.

We next describe the sense in which the eigenfunctions $\varphi_{h,t}^{(u)}$ are regarded as random variables in the deformation parameters $u \in {B}^{k}(\varepsilon)$.  Consider a cut-off function $\chi \in C^{\infty}_0({B}^{k}(\varepsilon);[0,1])$
 with  $\chi(u) = 1$ for $ u \in {B}^{k}(\varepsilon/2).$
We introduce the normalization constant
\begin{equation}\label{normalizing factor}
c_k(\varepsilon):= \left(\int_{{B}^{k}(\varepsilon)} \chi^2(u) du \right)^{-1},
\end{equation}
and define the probability measure $\nu$ on  ${B}^{k}(\varepsilon)$ by
$$ d\nu(u):= c_k(\varepsilon) \, \chi^2(u) \, du.$$
The introduction of the cut-off function in the definition of the probability measure is to ensure that all the integrands we consider, regarded as functions of $u$, are compactly supported in the interior of the ball, ${B}^{k}(\varepsilon)$.  This is crucial for the $\h$-microlocal characterization of the variance in Proposition \ref{A_x}.

 We view the real part of the perturbed
 eigenfunctions $\varphi_{h,t}^{(u)}$ defined in \eqref{perturbed efxs} as random variables
 $$\Re \left(\varphi^{(\cdot)}_{h,t}(x)\right):B^{k}(\varepsilon) \to \mathbb R$$
 depending on the  spatial parameters  $x\in M$.
 Since one can study the distribution of a random variable such as $\Re (\varphi^{(\cdot)}_{h,t}(x))$
 by understanding its moments,
 we dedicate this paper to study the asymptotics of the variance $Var [\Re (\varphi^{(\cdot)}_{h,t}(x))]$  and of the odd moments
 $\mathbb E [\Re(\varphi_{h,t}^{(\cdot)}(x))]^p$
 in the semiclassical limit $\hbar \to 0^+$. 
  \begin{remark}
Throughout the paper, we write that a condition holds \emph{locally uniformly} in a set $U$ whenever it holds
  uniformly on compact subsets of $U$. 
   \end{remark} 
Our first result holds for general Riemannian manifolds $(M,g_0)$.

\newpage
   \begin{theorem}\label{odd moments and variance}
	Let $(M,g_0)$ be a compact Riemannian manifold of dimension $n$ and
	let $E$ be a regular value of $p_0$. Suppose $g_u$ is
	 a perturbation of $g_0$ with $u \in B^k(\varepsilon)\subset \R^k$ that
is admissible at every $x \in M$. Fix a positive integer $\tilde p \in {\mathbb Z}^+.$ 
	Then, for $\varepsilon>0$ and $|t| >0$ sufficiently small, depending on $(M,g_0)$ and $\tilde p$, there is $h_0(t,\varepsilon) >0$ such that for $\h \in (0, \h_0(t,\varepsilon)]$ and $x \notin V^{-1}(E)$,\medskip
	 \begin{enumerate}
	\item \label{first} There exists a  constant $C>0$, independent of $h$,
	  with
	$$ Var \left[\Re \left(\varphi^{(\cdot)}_{h,t}(x)\right) \right] \leq C.$$
	\item For $p \in {\mathbb Z}^+$ odd with $p\leq \tilde p$, 
	$$\mathbb E \left[\Re(\varphi_{h,t}^{(\cdot)}(x))\right]^p= \mathcal O(\h^\infty).  $$
	\end{enumerate}	
These estimates are locally uniform for $x \notin V^{-1}(E).$ 
\end{theorem}\ \smallskip
\begin{remark} \label{upper}
We note that if $g_u$ is admissible for all  $x \in M,$ we prove that as $h \to 0^+,$
$$ \int_{B^k(\varepsilon)} |\varphi_{h,t}^{(u)}(x)|^2 d\nu(u) = {\mathcal O}(1),$$
provided $\varepsilon  >0$ and $|t| >0$ are sufficiently small. Moreover, this estimate is uniform in $x \in M.$  However, the statement of the corresponding result for the variance in Theorem \ref{odd moments and variance} (\ref{first}) requires the asymptotic vanishing of the mean (the first moment) for which we require the condition that $x \notin V^{-1}(E)$ (see (\ref{eq:variance})).
\end{remark}
The assumption $x \notin V^{-1}(E)$ in Theorem \ref{odd moments and variance} is used in the integration by parts argument in (\ref{ibp}) to estimate the odd moments. At present, we do not know whether the  estimates for odd moments away from these generalized turning points extend uniformly to all $x \in M.$
This assumption is vacuous in the homogeneous case $V=0$. \\

If the metric perturbation $g_u$ is admissible, there exist $c>0$ and an  $n$-tuple of  $u$-coordinates   denoted by
 $u'=(u_1, \dots, u_n) \in B^{n}(\varepsilon)$
for which $|d_{u'}d_{\xi} p_{(u',u'')}(x ,\eta)| \neq 0$ at $u=0$  provided
$(x,\eta) \in p_0^{-1}(E-c\varepsilon, E+c\varepsilon)$. Using this, we show via an Implicit Function Theorem argument that  
one can locally parametrize $u'$ as a smooth function of $(y,  \eta)\in p_0^{-1}(E)$,  $u'=u'( y, \eta)$  (see \eqref{u'}).
We write $u''\in B^{k-n}(\varepsilon)$ for the omitted parameters  and the dependence of $u'(y,\eta)$ on $(u'',x)$ as parameters is understood.  Furthermore, without loss of generality, we assume that the coordinates of $u$ are ordered so that $u=(u',u'')$.\\

%for the relevant generating function $S(t,u,y,\eta)$ in (\ref{Taylor for S}) and  for  points $(u',\tau; y,\eta) \in \Gamma_{t,x,u''}$ 
%one can locally parametrize $u'$ as a smooth function of $(y,  \eta)\in p_0^{-1}(E)$,  $u'=u'( y, \eta)$; here $\Gamma_{t,x,u''}$ is 
%the relevant Lagrangian submanifold of $T^* B^n(\varepsilon) \times T^*M$ defined in (\ref{Lagrangian}). 
%We write $u''\in B^{k-n}(\varepsilon)$ for the omitted parameters  and the dependence of $u'(y,\eta)$ on $(u'',x)$ as parameters is understood.  Furthermore, without loss of generality, we assume that the coordinates of $u$ are ordered so that $u=(u',u'')$.\\

Let  $H_{p_u}$ be the Hamiltonian vector field of $p_u \in C^{\infty}(T^*M)$ and denote by
$G_{u}^{s}:S_{g_u}^*M \to S_{g_u}^*M$  the bicharacteristic flow associated to $H_{p_u}$ at time $s$.
In the case where the manifold $(M,g_0)$ has an ergodic geodesic flow $G_0^s:S^*M \rightarrow S^*M$, we get asymptotic results for the variance provided we consider quantum ergodic sequences of eigenfunctions (for a precise definition see \eqref{QE}).
We continue to write $\chi$ for the cut-off function in the definition of the probability measure and
$c_k(\varepsilon)$ for the corresponding normalizing factor \eqref{normalizing factor}.\\

\begin{theorem}\label{moments function}
	Let $(M,g_0)$ be a compact Riemannian manifold of dimension $n$ and
	let $E$ be a regular value of $p_0$. Assume the geodesic flow on $p_0^{-1}(E)$ is ergodic
	and that $\{\varphi_\h\}_{\h \in (0,\h_0]}$ is a quantum ergodic sequence of $L^2$-normalized
	eigenfunctions of $P_0(\h)$. Suppose $g_u$ with $u\in B^{k}(\varepsilon)$ is
	a perturbation of $g_0$ that is admissible at $x \notin V^{-1}(E)$. Fix  a positive integer $\tilde p\in \mathbb Z^+.$ 
	Then, for $|t|>0$ and $\varepsilon >0$ sufficiently small, depending on $(M,g_0)$ and $\tilde p,$ \\
	
	 \begin{enumerate}
	\item
	$ \lim _{\h \to 0^+} Var  \left[\Re(\varphi_{h,t}^{(\cdot)}(x))\right]=
		 \displaystyle{
		\int_{ B^{k-n}(\varepsilon)}\beta^{k}_{t,x}(u'')\, du''}$,\\
		
	\noindent where $\beta^{k}_{t,x}:B^{k-n}(\varepsilon) \to \R$ is  defined by
	 	\begin{equation}\label{beta} \displaystyle
	  	\beta^{k}_{t,x}(u''):=
	 	\frac{c_k(\varepsilon) (1+ {\mathcal O}(t)) }{|t|^n |p_0^{-1}(E)|}
	 	\underset{p_0^{-1}(E)}{\int}
	 	\frac{ |\det(d_x \pi G_{(u'(y,\eta),u'')}^{-t}( x, \eta))|}{|\det(d_{u'} d_{\xi}\; p_{(u'(y,\eta),u'')}(x ,\eta))|}
	 	\, \chi^2(u'(y,\eta),u'') \;d\omega_E(y,\eta)
	 	\end{equation}
	and $u'=u'( y,  \eta)$ is defined in \eqref{u'}.\\
	 \item For $p \in {\mathbb Z}^+$ odd with $p\leq \tilde p$,
	 $$\lim _{\h \to 0^+} \;\mathbb E \left[\Re(\varphi_{h,t}^{(\cdot)}(x))\right]^p=0.$$
		 \end{enumerate} 
\end{theorem}
\

\subsection{Motivation}
We proceed to describe two ideas that motivate our work. We first explain how the  underlying ideas in our approach are motivated by the random wave conjecture. We then relate our results to the physics notion of Loschmidt echo. \

%-----------------------------------------------------------------------------------------------
\subsubsection*{Random wave conjecture}
In 1977 M. Berry conjectured that  the real and imaginary parts of the
eigenfunctions $\varphi_\h$ in the chaotic case  resemble random waves, \cite{Ber}.
It is also believed that the eigenfunctions $\varphi_\h$  of quantum
mixing systems behave locally as independent gaussian variables as $\h\to 0$;
see for example the discussion in \cite{HR} and references therein.
 One of the common issues is to define 
 a probability model where the random functions mimic the chaotic eigenfunctions.
 This is the role we give
to the perturbations $\varphi_{h,t}^{(u)}$.

%-----------------------------------------------------------------------------------------------
\subsubsection*{Loschmidt echo}

A natural way of measuring the noise affecting a given system is the Loschmidt echo.
The idea behind this concept is to measure the sensitivity of quantum evolution to perturbations, by
propagating  forward an initial state $\psi$ using the unperturbed hamiltonian $p_0$, and propagating it
back via the perturbed one $p_u$  after time $t$. Thus, the objects
of interest in this case are the states $e^{\frac{it}{\h} P_u(\h)}e^{-\frac{it}{\h} P_0(\h)} \psi$ and the  \emph{Loschmidt Echo}, $M_{LE}(t)$,
is defined to be the return probablility to the initial state:
$$M_{LE}(t)=\left|\langle e^{-\frac{it}{\h} P_u(\h)}e^{\frac{it}{\h} P_0(\h)} \psi, \psi\rangle \right|^2.$$

   \begin{center}
           \parbox{2.5cm}{\fbox{\includegraphics[height=2.5cm]{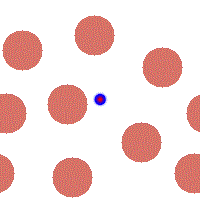}}}
           \hspace{0.2cm}
           $\xrightarrow[\hspace{1cm}]{e^{-\frac{it}{\h}P_0(\h)}}$
           \hspace{0.2cm}
           \parbox{2.5cm}{\fbox{\includegraphics[height=2.5cm]{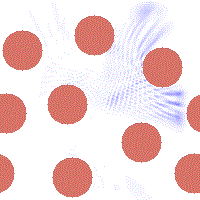}}}
           \hspace{0.2cm}
           $\xrightarrow{e^{\frac{it}{\h} P_u(\h)}}$
           \hspace{0.2cm}
           \parbox{2.5cm}{\fbox{\includegraphics[height=2.5cm]{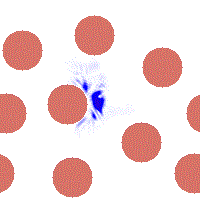}}}
     \end{center}

\begin{center}\ \\
\footnotesize{Illustration of the state of particle initially placed in the center of a square billiard with an irregular array of $10$ circular scatterers with initial momentum pointing to the left \cite{CPW}.}
\end{center}\ \\
\
We are interested in the case when the initial state $\psi$ is an eigenfunction, $\psi=\varphi_\h$. In this simpler case
$M_{LE}(t)$ is called the \emph{survival probability} \cite{WC} and we have
$$e^{\frac{it}{\h} P_u(\h)}e^{-\frac{it}{\h} P_0(\h)} \varphi_h= e^{-\frac{it E(\h)}{\h}} \varphi_{h,t}^{(u)}.$$
To be precise, for an initial state $\varphi_\h$
 the Loschmidt Echo is simply
	$$M_{LE}(t)=\left|\langle \varphi_{h,t}^{(u)}, \varphi_\h \rangle \right|^2.$$
 As the definition shows,
 the fidelity $M_{LE}(t)$ can be interpreted as the decaying overlap between the evolution
 $\varphi_{h,t}^{(u)}$ and the unperturbed evolution $ \varphi_h$,
 \cite{JP, JSB, Per}.\\

 In recent work \cite{ET},  Eswarathasan and Toth have proved related results for \emph{magnetic} deformations of the Hamiltonian $p_0(x,\xi) = |\xi|_{g_0(x)}^2 + V(x)$. We extend their upper bound results
 to large families of metric deformations. In additon, we characterize the asymptotic results in terms of variance and show that all odd moments are negligible. Although we do not have a rigorous argument at the moment, we hope that by further developing the methods of the present paper, we will be able to  compute the higher
even moments $\lim _{\h \to 0^+} \;\mathbb E [\Re(\varphi_{h,t}^{(\cdot)}(x))]^{2p}$ for $ p \geq 2,$  and compare them with the Gaussian prediction of the random wave model. We plan to return to this question elsewhere.

%-----------------------------------------------------------------------------------------------

\subsection{Outline of the paper}\ \\

In Section \ref{section: setting} we introduce the background material and notation from semiclassical analysis that we shall use to prove our results. We first show that the perturbations are semiclassically localized in $p_0^{-1}(E)$ and then explain
how to microlocally cut off the propagator $e^{-\frac{it}{\h} P_u(\h)}$ to obtain a localized
approximation of $\varphi_{h,t}^{(u)}$.\\ %The material here is standard \cite{Zw} but we have included it here for the benefit of the reader.\\

In Section \ref{section: ell>0}, we study the odd moments of  $\Re(\varphi_{h,t}^{(\cdot)}(x))$.
 Provided the metric perturbation satisfies part \eqref{cond 2} of the admissibility condition at $x_*\in M$,
  we prove in Lemma \ref{odd moments} that for fixed $\tilde p \in \mathbb Z^+$ 
and  $\ell, q \in \mathbb Z^+$ with $1 \leq \ell \leq \tilde p$ and $2q \leq \tilde p$,
	 $$\int_{B^k(\varepsilon)} \left(\varphi_{h,t}^{(u)}(x)\right)^\ell
	 	   \left|\varphi_{h,t}^{(u)}(x)\right|^{2q} \;d \nu(u)  = \mathcal O (\h^\infty)
	 	   \qquad \quad  \text{as}\;\; \h\to 0^+,$$
 for   $\varepsilon>0$ and $|t| >0$ sufficiently small depending on $(M,g_0)$ and $\tilde{p}$.
The error is locally uniform in $x \in \mathcal W \cap (V^{-1}(E))^c$ where $\mathcal W$ is the open neighborhood of $x_*$ given by part \eqref{cond 2} of the admissibility condition.

Using  Lemma \ref{odd moments}  and the binomial expansion for $(\varphi + \bar \varphi)^p=(2 \Re \varphi)^p$,
we prove that  for $p \in \mathbb Z^+$ odd,
\begin{equation}\label{negligible odd 1}
\mathbb E \left[ \Re(\varphi_{h,t}^{(\cdot)}(x)) \right]^p= \mathcal O(\h^\infty),
\end{equation}
locally uniformly
 in $\mathcal W \cap (V^{-1}(E))^c$.\\

In Section  \ref{section: variance}  we study the variance of  $\Re(\varphi_{h,t}^{(\cdot)}(x)).$
Provided the perturbation is admissible at $x \in (V^{-1}(E))^c$,
the case  $p=1$ in \eqref{negligible odd 1} shows that our variables are semiclassically centered in the sense
$$\mathbb E \left[ \Re(\varphi_{h,t}^{(\cdot)}(x)) \right]= \mathcal O(\h^\infty).$$
Therefore,
\begin{equation} \label{eq:variance}
Var \left[ \Re(\varphi_{h,t}^{(\cdot)}(x)) \right] =
\int_{B^k(\varepsilon)} | \varphi_{h,t}^{(u)}(x)|^2\,d\nu(u)  +\mathcal O(\h^\infty).
\end{equation}
It follows that studying the variance is equivalent to understanding the behavior of the right hand side in the
previous equality.  In Proposition \ref{A_x} we compute the asymptotics of the RHS in (\ref{eq:variance}) and consequently prove Theorems \ref{odd moments and variance} and \ref{moments function}.\\

In Section \ref{section: admissible perturbations} we show that  there always exist large families
of admissible perturbations. We show that the notion of admissibility is related to having sufficiently many volume preserving directions in which the metric tensor $g_0$ is perturbed (this is condition A), and to having at least one direction in which $g_0$ is conformally perturbed (this is condition B).\\

\begin{remark} \label{restriction} We note here that there is an easy consequence of Theorem 1 (see also Remark \ref{upper}) that concerns restriction bounds of $\varphi_{h,t}^{(u)}$ to smooth submanifolds $H \subset M$ under the assumption that the family $g_u$ is admissible for all $x \in M.$ 
Indeed, since our upper bounds   are then  uniform in $x \in M$ (see Remark \ref{upper}), by integrating over $H$ and applying Fubini, one gets that for $\h \in (0,\h_0],$ there exist a constant $C = C(H,h_0) >0$ with
$$\int_{B^k(\varepsilon)} \int_H |\varphi_{h,t}^{(u)}(s)|^2 \, d\sigma_H(s) d \nu(u) \leq C.$$
By the Tschebyshev inequality, it then follows that for {\em any} sequence $\omega(\h) = o(1)$ as $\h \to 0^+,$ there is a measurable  $D(h) \subset B^k(\epsilon)$ with $\lim_{h\to 0^+} \frac{|D(h)|}{|B^k(\epsilon)|} = 1$ such that for $u \in D(h),$
$$\int_H |\varphi_{h,t}^{(u)}(s)|^2 d\sigma_H(s) = {\mathcal O}( |\omega(\h)|^{-1}).$$
Therefore, the restriction bounds for most perturbed eigenfunctions are much smaller  than the universal bounds for $ \int_H |\varphi_{\h,t}^{(0)}(s)|^2 d \sigma_H(s)$ in \cite[Theorem 3]{BGT} and tend to be consistent with the ergodic case \cite{DZ,TZ}.
\end{remark}

%----------------------------------------------------------------------------------------------------------------------------------------------------
\subsection{Acknowledgement}
 The authors would like to thank the referee for many helpful detailed comments on the manuscript and for pointing out an error in the previous version of the paper.

%----------------------------------------------------------------------------------------------------------------------------------------------------
\section{Background and Notation}\label{section: setting}

 In this section we introduce
some background material on eigenfunction localization and semiclassically cut off propagators. Most of this is standard in semiclassical analysis, but we include it for the benefit of the reader. We refer to \cite{Zw} for further details.\\
Let $M$ be a compact Riemannian manifold of dimension $n$.
We work with the class of semiclassical symbols
\begin{align*}
S^{m,k}_{cl} (T^*M) :=  \Big\{ a \in C^{\infty}(T^*M \times& (0,\h_0] ):  \;\;
a(x,\xi;\h) \sim_{\h \to 0^+} \h^{-m} \sum_{j=1}^\infty a_j(x,\xi) \,\h^j\;   \\
&\text{with}\;\;\;
  |\partial^\alpha_x \partial^\beta_\xi a_j(x,\xi)| \leq C_{\alpha, \beta}\, (1+|\xi|^2)^{\frac{k - |\beta|}{2}} \Big \}.
\end{align*}
For $a \in S^{m,k}_{cl} (T^*M) $, we consider the Schwartz kernel in $M \times M$  locally  of the form
$$Op_\h (a) (x,y) = \frac{1}{(2\pi \h)^n} \int_{\R^n} e^{\frac{i}{\h} \langle x-y, \xi \rangle} a(x,\xi;\h) \,d\xi,$$
for $(x,y) \in U \times V$ where $U,V \subset \R^n$ are local coordinate charts.
The corresponding space of pseudodifferential operators is defined to be
$$\Psi^{m,k}_{cl} (M):= \{ Op_\h (a): \;\;a \in S^{m,k}_{cl} (T^*M)\}.$$
Let $N$ be another compact  n-dimensional  Riemannian manifold.
We also consider the class of Fourier integral operators $I_{cl}^{m,k}(M \times N, \Gamma)$
with Schwartz kernels defined in the form
$$F_\h (x,y) = \frac{1}{(2\pi \h)^n} \int_{\R^n} e^{\frac{i}{\h} \phi(x,y, \xi )} a(x, y,\xi;\h) \,d\xi$$
for $(x,y) \in U \times V$ where $U,V \subset \R^n$ are local coordinate charts
and $a \in C^\infty_0(U\times V \times \R^n \times (0,\h_0])$ with
$a(x, y,\xi;\h) \sim_{\h \to 0^+} \h^{-m} \sum_{j=1}^\infty a_j (x, y,\xi) \,\h^j$.
Here $\phi$ denotes a non-degenerate phase function in the sense of
H\"{o}rmander \cite[Def (2.3.10)]{Dui} and $\Gamma$ is an immersed Lagrangian submanifold of $T^*M \times T^*M$ with
$$\Gamma=\{ (x,d_x \phi,\; y, -d_y \phi):\;\; d_\xi \phi (x,y,\xi)=0 \} \subset T^*M \times T^*N.$$
Finally, throughout the manuscript we say that a sequence of $L^2$-normalized eigenfunctions $\{\varphi_{\h_j}\}_{j\geq 1}$ of $P_0(\h_j)$ with $P_0(\h_j) \varphi_{\h_j} = E(\h_j) \varphi_{\h_j}$ and $E(\h_j) = E + o(1)$ is quantum ergodic (QE) if for any $a(x,\xi,\h) \sim \sum_{k=0}^{\infty}a_k(x,\xi) \h^k \in S^{0,0}_{cl}(T^*M \times [0,h_0)),$

\begin{equation}\label{QE}
 \langle Op_{\h_j}(a) \varphi_{\h_j}, \varphi_{\h_j} \rangle\;
  \underset{j \to \infty}{\longrightarrow}
  \;\int_{p_0^{-1}(E)} a_0(x,\xi) d\omega_E(x,\xi)
\end{equation}
where, $d\omega_E$ is normalized Liouville measure on $p_0^{-1}(E).$

%-------------------------------------------

\subsection{Eigenfunction localization}\label{localization}
 Fix  $t \neq 0$ and let $E$  be a regular value of $p_0.$ Then, for  $u\in B^k(\varepsilon)$  we introduce the cut-off
 functions  on $T^*M$ 
 \begin{equation}\label{chi}
 \chi_{E}^{(u)}(x,\xi)=\chi_0 \left( p_0(G^{-t}_u(x,\xi))-E \right),
\end{equation}
where $\chi_0 \in C^\infty_0 ([-\epsilon,\epsilon]; [0,1])$ equal to $1$ on $[-\epsilon/2,\epsilon/2].$  Consequently, 
\begin{equation} \label{suppfunction1}
\text{supp} \, \chi_{E}^{(0)} \subset p_{0}^{-1}( [E-\varepsilon, E+\varepsilon]). \end{equation}
 Since $\chi_E^{(u)}(x,\xi)=\chi_0(p_0(x,\xi)-E+ \mathcal O (|u|))$, the support of $\chi_E^{(u)}$ remains localized near
the hypersurface $p_0^{-1}(E)$ for all $u \in B^k(\varepsilon)$  (see (\ref{constant}) below for precise control) and that $\varphi_{h,t}^{(u)}$ is a normalized eigenfunction of the operator
	$$Q_u(\h):=e^{-\frac{i}{\h} t P_u(\h)} P_0(\h) e^{\frac{i}{\h}t P_u(\h)} \in \Psi_{cl}^{0,2}(M)$$
with eigenvalue $E(\h)$.
By Egorov's Theorem
$Q_u(\h)=Op_\h(p_0 \circ G_u^{-t}) +\mathcal O_{L^2\to L^2}(\h)$, and since  $E(\h) \in [E- o(1),E+ o(1)]$, it follows that $(Q_u(\h)-E)\varphi_{h,t}^{(u)}=o(1)$. Using that $Q_u(\h)$ is
$\h$-elliptic off $(p_0 \circ G_u^{-t})^{-1}(E)$, a parametrix construction
\cite[Thm. 6.4] {Zw} gives
$\| \varphi_{h,t}^{(u)} - Op_\h(\chi_E^{(u)}) \varphi_{h,t}^{(u)}\|_{L^2}= \mathcal O(\h^\infty)$ and therefore
 $WF_\h(\varphi_{h,t}^{(u)})\subset(p_0 \circ G_u^{-t})^{-1}(E)$.
  Fix $t \neq 0$ and $\epsilon >0$ small.   Given $x \in M$, there is a coordinate chart $U$ with $x \in U$ and such that $\pi( G^{-t}_u(x,\xi)) \subset U$ for $u \in B^k(\varepsilon).$ Then, since $p_0(G_0^{-t}(x,\xi)) = p_0(x,\xi),$  by Taylor expansion around $u=0,$ 
  \begin{align} \label{taylor expansion}
p_0 ( G_u^{-t} (x,\xi) ) &= p_0(x,\xi) + R_1(x,\xi;u,t)
  \end{align}
 where 
 $$ |R_1(x,\xi,u;t)| \leq \sqrt{k} \max_{u \in B^k(\varepsilon)}   \| d_u  (p_0  \circ G_u^{-t})(x,\xi)\| \,|u|.$$
 
 Here, given a arbitrary matrix $A= (a_{ij})$ we write $\|A\|:= \max_{i,j} |a_{ij}|.$
 
 For $t \neq 0$ fixed, we define the constant
 \begin{equation} \label{constant}
 c:= k \max_{u \in B^{k}(\varepsilon)}  \max_{(p_0 \circ G_u^{-t})^{-1}(E)}  \| d_u  ( p_0 \circ G_u^{-t})(x,\xi)\|.\end{equation}
 It then follows from (\ref{taylor expansion}) that  
\begin{equation} \label{energyshell}
(p_0 \circ G_u^{-t})^{-1}(E)\subset p_0^{-1}(E-c\|u\|, E+c\|u\|)) \end{equation}
 for $c>0$  in (\ref{constant}) and  $u \in B^k(\varepsilon).$ 
 Consequently,
$WF_\h(\varphi_{h,t}^{(u)}) \subset p_0^{-1}(E-c\varepsilon, E+c\varepsilon)$ and 
by a Sobolev lemma argument one can also prove
 $\| \varphi_{h,t}^{(u)} - Op_\h(\chi_E^{(u)}) \varphi_{h,t}^{(u)}\|_{C^k}= \mathcal O_{C^k}(\h^\infty)$.
It follows that
 \begin{equation}\label{approx of varphi}
 \varphi_{h,t}^{(u)}=Op_\h( \chi^{(u)}_E) \circ e^{-\frac{it}{\h} P_{u}(\h)}  \circ Op_\h( \chi^{(0)}_E) \;\varphi_\h+ \mathcal O_{C^k}(\h^\infty),
 \end{equation}
 and from (\ref{energyshell}) and (\ref{suppfunction1}), with $c>0$ in (\ref{constant}),
 \begin{equation} \label{suppfunction2}
 \text{supp} \, \chi_{E}^{(u)} \subset p_{0}^{-1}([E-\epsilon - c\|u\|, E+\epsilon + c\|u\|]). \end{equation} 

%---------------------------------
\subsection{Semiclassically cut off propagators}

Motivated by the approximation \eqref{approx of varphi}, for $\h \in (0, \h_0]$, $u\in B^k(\varepsilon)$ and $|t|>0$ small, we define the semiclassically cut off
Fourier integral operators $W_{t,u}(\h)\in I_{cl}^{0, -\infty}(M \times M, \Gamma_{u,t}),$ 
\begin{equation}\label{W_{t,u}}
	W_{t,u}(\h):=Op_\h( \chi^{(u)}_E) \circ e^{-\frac{it}{\h} P_{u}(\h)}  \circ Op_\h( \chi^{(0)}_E),
\end{equation}
with  immersed Lagrangian,
\
\begin{align*}
\Gamma_{u,t}=& \left\{(x,\xi; y, \eta): \; (x,\xi)=G_u^{-t}(y,\eta), % \in supp \,\chi_E^{(u)} \;\;\; \text{and}\;
 \quad (y,\eta) \in supp \,\chi_E^{(0)}  \right\}%\\ &\hspace{10cm} 
 \subset T^*M \times T^*M.
\end{align*}
\
We note that since $G^{-t}_u$ is then a 
symplectomorphism that is close to the identity,
 there exists a local generating function $S(s, u, \xi; x)$
with $(x,d_xS(s,u,\eta;x))=G_u^{-t}(d_\eta S(s,u,\eta;x,\eta),\eta)$ for $s$ close to $t$.
It follows that
\
\begin{align}\label{Gamma_{u}}
\Gamma_{u,t}
&= \left\{\left(x,d_xS(t, u, \eta; x) ;\; d_\eta S(t, u, \eta; x), \eta \big.\right)
\in supp \,\chi_E^{(u)} \times supp \,\chi_E^{(0)} \right\}  \notag \\
 &\hspace{10cm} \subset T^*M \times T^*M.
\end{align}
\
The generating function $S(s,u,\eta;x)$ solves  the Hamilton-Jacobi initial value problem
	$$\begin{cases}
	\partial_s S(s,u,\eta;x)+ p_u(x,d_x S(s,u,\eta;x))=0, \\
	S(0,u,\eta;x)=\langle x, \eta \rangle,
 	\end{cases}$$
and therefore, a Taylor expansion in $s$ around $s=0$ gives
\begin{equation}\label{Taylor for S}
S(s,u,\eta;x)=\langle x,\eta \rangle -s \,p_u(x,\eta) + \mathcal O(s^2).
\end{equation}

Given local coordinate charts $U,V \subset \R^n$  consider the local phase function
$\phi_t \in C^\infty(V \times B^k(\varepsilon) \times \R^n)$,
	\begin{equation}\label{phase function}
	\phi_t(y,u,\xi;x):= S(t, u, \xi; x )-\langle y, \xi \rangle.
	\end{equation}
%for $u \in U \times B^{k}(\varepsilon)$. 
The Schwartz kernel of $ W_{t,u}(\h)$  is locally  of the form
	\begin{equation}\label{kernel of W_{t,u}}
	 W_{t,u}(\h)(x,y)= \frac{1}{(2\pi h)^n} \int_{\R^n} e^{\frac{i}{\h} \phi_t(y,u, \xi; x) } a_t(u,y,\xi;x,\h) \; d\xi + K_x(y,u),
	\end{equation}
where $|\partial^\alpha_x \partial^\beta_y K_x(y,u)|= \mathcal O_{\alpha,\beta} (\h^\infty)$ uniformly for
$(y,x,u) \in V \times U \times B^k(\varepsilon).$ \\
The amplitude $ a_t(u,y,\xi;x,\h) \sim \sum_{j=0}^\infty a_{t,j}(u,y,\xi;x)\h^j$ with
$$a_{t,j}(u, \cdot ,\cdot \,; \cdot) \in
 C^\infty( B^k(\varepsilon), C_0^\infty(V \times \R^n \times U)).$$
 \noindent  From (\ref{energyshell}) it is clear that the support of $\chi_E^{(u)}$ remains localized near
 the hypersurface $p_0^{-1}(E)$ for all $u \in B^k(\varepsilon);$ indeed, with the constant $c>0$ in (\ref{constant}),
 %$$ supp \chi_{E}^{(u)} \subset \{ (x,\xi) \in T^*M; \}$$
 % \cs add more detail about energy cutoff here \cs
 \begin{align}\label{amplitudes}
 supp\;(a_t(u, \cdot, \cdot \,; \cdot , \h))  \subset \notag 
 \{(y,\xi,x) \in T^*U \times V: &(y,\xi)\in \text{supp} \chi_E^{(0)} \subset p_0^{-1}(E-\varepsilon, E+\varepsilon),\\
 & y = d_{\xi}S(t,u,\xi;x) = x + {\mathcal O}_{u}(t) \}.
 \end{align}
 We note that since $t \neq 0$ is a fixed small parameter and $u \in B^k(\varepsilon)$ with $\epsilon >0$ small, it follows from (\ref{amplitudes}) that  the amplitudes $a_{t,j}(u,\cdot,\cdot;\cdot)$ are supported near the set $\{(y,\xi,y) \in T^*U \times U\}.$ Similarily, the Lagrangian manifolds $\Gamma_{t,u} \subset T^*M \times T^*M$ are localized near the diagonal $\Delta_{T^*M \times T^*M} = \{ (x,\xi;x,\xi) \in T^*M \times T^*M \}.$ 

%-------------------------------------------------------------------------------------------------------

\section{Odd moments}\label{section: ell>0}
The purpose of this section is to show that provided the metric $g_0$ is conformally deformed in at least one
 direction, its odd moments  are negligible for general geodesic flows.
Throughout this section we continue to assume that $(M,g_0)$ is a compact Riemannian manifold and $E$ is a regular value of $p_0$.
We prove
\begin{proposition}\label{negligible odd moments}
Let $g_u$ with $u \in B^k(\varepsilon)$ be a perturbation of $g_0$
that satisfies part \eqref{cond 2} of the admissibility condition at  $x_*\in M$.  Fix $\tilde{p} \in {\mathbb Z}^+$ and suppose $p  \in {\mathbb Z}^+$ is odd with $p \leq \tilde{p}.$ Then, for $\varepsilon>0$ and $|t| >0$ sufficiently small depending on $(M,g_0)$ and $\tilde{p},$  
%there exists $\tau_{\varepsilon,t}>0$   such that for $\varphi_{\h}^{(u)} = e^{-it P_{u}(\h)/h} \varphi_{\h},$
\begin{equation}\label{negligible odd}
\mathbb E \left[ \Re(\varphi_{h,t}^{(\cdot)}(x)) \right]^p= \mathcal O(\h^\infty) \qquad \quad  \text{as}\;\; \h\to 0^+,
\end{equation}
locally uniformly for $x \in \mathcal W \cap (V^{-1}(E))^c.$ Here $\mathcal W$ is the open neighborhood of $x_*$ given by part \eqref{cond 2} of the admissibility condition. \end{proposition}

\begin{proof}
 Given $p \leq \tilde p$ odd,
$$\mathbb E \left[\Re(\varphi_{h,t}^{(\cdot)}(x))\right]^p= \int_{B^k(\varepsilon)} \left(\Re(\varphi_{h,t}^{(u)}(x))\right)^p \;d \nu(u),$$
and for any complex $\varphi$ the binomial expansion of $(\varphi + \bar \varphi)^p=(2 \Re \varphi)^p$ for $p$ odd
gives
	\begin{equation}\label{binomial expansion}
	(\Re \varphi)^p
	 	= %\frac{1}{2^p}\binom{p}{p/2}\, |\varphi|^p+
	 	\frac{1}{2^p}\sum_{0\leq j <\frac{p}{2}} \binom{p}{j} \varphi^{p-2j}\, |\varphi|^{2j}
	 	+\frac{1}{2^p}\sum_{\frac{p}{2}< j \leq p} \binom{p}{j} \bar \varphi^{\;2j-p}\, |\varphi|^{2(p-j)}.
	\end{equation}

Therefore, to prove Proposition \ref{negligible odd moments}, it suffices to show that %there exists $\tau_{\varepsilon,t}>0$ such that
\begin{equation}\label{moments}
	\int_{B^k(\varepsilon)} \left(\varphi_{h,t}^{(u)}(x)\right)^\ell
	\left|\varphi_{h,t}^{(u)}(x)\right|^{2q} \;d \nu(u)= \mathcal O(\h^\infty) \quad \quad \text{for }\;\;\, 1 \leq \ell \leq \tilde p, \, \; 2q \leq \tilde p, \;
\end{equation}
locally uniformly in $x \in \mathcal W\cap (V^{-1}(E))^c $ as $\h\to 0^+$.

Since the proof of \eqref{moments} is somewhat technical, we prove it  separately as Lemma \ref{odd moments}.
Combining \eqref{moments} with the binomial expansion \eqref{binomial expansion}
 completes the proof.
\end{proof}

\ \\
We have reduced the proof of Proposition \ref{negligible odd moments} to establishing the following Lemma.

\begin{lemma}\label{odd moments}
	Let $g_u$ with $u \in B^k(\varepsilon)$ be a perturbation of $g_0$
that satisfies part \eqref{cond 2} of the admissibility condition at $x_*\in M$. Fix $\tilde p \in \mathbb Z^+$ 
and suppose  $\ell, q \in \mathbb Z^+$ with $1 \leq \ell \leq \tilde p$ and $2q \leq \tilde p$.
Then, for   $\varepsilon>0$ and $|t| >0$ sufficiently small depending on $(M,g_0)$ and $\tilde{p}$,
%$ there exists $\tau_{\varepsilon,t}>0$ such that
	 $$\int_{B^k(\varepsilon)} \left(\varphi_{h,t}^{(u)}(x)\right)^\ell
	 	   \left|\varphi_{h,t}^{(u)}(x)\right|^{2q} \;d \nu(u)  = \mathcal O (\h^\infty)
	 	   \qquad \quad  \text{as}\;\; \h\to 0^+,$$
locally uniformly for $x \in \mathcal W \cap (V^{-1}(E))^c$. Here $\mathcal W$ is the open neighborhood of $x_*$ given by part \eqref{cond 2} of the admissibility condition.
	\end{lemma}

\begin{proof}
 We identify the product manifold 
 $M^{(\ell+2q)}$ with
$M^{(\ell)}\times M^{(q)}\times M^{(q)}$ and write 
$(\tilde y, \tilde z, \tilde z'):=(y ^{(1)}, \dots,  y^{(\ell)},   z ^{(1)}, \dots,  z^{(q)}, z'^{(1)}, \dots,  z'^{(q)})  \in V^{(\ell+2q)}$ for the local coordinates .\ \\

 By assumption, there exists %$0<\tau_{\varepsilon,t}<inj(M)$ and
  $a \in C^\infty(M)$ so that $\delta_{u_\alpha}g_u^{-1}(x)=a(x)\,g_0^{-1}(x)$
  with $a(x) \neq 0$ for all $x \in \mathcal W$.
 Let us continue to write $\chi$ for the cut-off function appearing in the definition
  of the probability measure $\nu$, and $c_k(\varepsilon)$ for the corresponding normalizing factor in \eqref{normalizing factor}.
Since from \eqref{approx of varphi},
 $\varphi_{h,t}^{(u)}(x)=[W_{t,u}(\h)\varphi_\h](x)+ \mathcal O(\h^\infty),$
 writing $W_{t,u}(x,y)$ for the kernel of $W_{t,u}$  we get
\begin{align}\label{odd moments computation}
	& \int_{B^k(\varepsilon)}  \left(\varphi_{h,t}^{(u)}(x)\right)^\ell
	   \left|\varphi_{h,t}^{(u)}(x)\right|^{2q} \;d \nu(u)  = \nonumber\\
	& =\underset{  B^{k}(\varepsilon)} {\int}
	   \left([W_{t,u}(\h)\varphi_\h](x) \big.\right)^\ell
	   	   \left|[W_{t,u}(\h)\varphi_\h](x)\big.\right|^{2q} \;d \nu(u)  + {\mathcal O}(\h^{\infty})\notag\\
%	& =\underset{ B^{k}(\varepsilon)}{\int}  \underset{M^{\ell +2q}} {\int}
%	\underset{1\leq j \leq q}{\prod_{1\leq i \leq \ell}} W_{t,u}(x,y^{(i)})
%	 W_{t,u}(x,z^{(j)})\overline{W_{t,u}(x,z'^{(j)})}
%	 \varphi_\h (y^{(i)})\varphi_\h (z^{(j)})\overline{\varphi_\h (z'^{(j)})}
%	 d\tilde y  d\tilde z d\tilde z' d \nu(u) \nonumber\\
	& =c_k(\varepsilon) \underset{ B^{k}(\varepsilon)}{\int}  \underset{M^{\ell +2q}} {\int}
		B_{t,u}^{[\ell, q]}(\tilde y, \tilde z, \tilde z' ; x, \h) 
		 \varphi_\h (y^{(i)}) \varphi_\h (z^{(j)}) \overline{\varphi_\h (z'^{(j)})}  \chi^2(u)
		 d\tilde y d\tilde z d\tilde z' du + {\mathcal O}(\h^{\infty}),
\end{align}

where
  $B_{t,u}^{[\ell, q]} \in C^\infty(V^{(\ell+2q)}  \times U \times [0, \h_0))$   is defined by the formula
$$B_{t,u}^{[\ell, q]}(\tilde y, \tilde z, \tilde z' ; x, \h)
	:=   \underset{1\leq j \leq q}{\prod_{1\leq i \leq \ell}} W_{t,u}(x,y^{(i)})
		 W_{t,u}(x,z^{(j)})\overline{W_{t,u}(x,z'^{(j)})}.$$

From  \eqref{kernel of W_{t,u}} we deduce the kernel expansion
\begin{align}\label{kernel B}
	&B_{t,u}^{[\ell, q]} (\tilde y, \tilde z, \tilde z'; x, \h) =   \notag\\
	& =\frac{ 1}{(2 \pi \h)^ {n(2q +\ell)}}
	\;\underset{\R^{nq}}{\int}  \underset{\R^{nq}}{\int}   \underset{\R^{n\ell}}{\int}
	   e^{\frac{i}{\h} \Phi^{[\ell,q]}_t(\tilde y, \tilde z, \tilde z', u, \tilde \xi, \tilde \eta, \tilde \eta'  ;x)} \;
	    c^{[\ell,q]}_t(u,\tilde y, \tilde z,\tilde z',  \tilde \xi, \tilde \eta, \tilde \eta'; x, \h)\;  d\tilde \xi\, d\tilde \eta \,d\tilde \eta '\;\, \notag\\
	 &\qquad + K_{x}(\tilde y, \tilde z, \tilde z', u) ,
\end{align}
for $ \Phi^{[\ell,q]}_t$, $c^{[\ell,q]}$ and $K_x$ as follows:\\

\noindent \emph{(i)}\; The \underline{phase function $\Phi^{[\ell, q]}_t$} is defined by
		\begin{align}\label{Phi}
			&\Phi^{[\ell, q]}_t(\tilde y, \tilde z, \tilde z', u, \tilde \xi, \tilde \eta, \tilde \eta'  ;x):= \notag \\
			 &\quad=  \sum_{j=1}^\ell \phi_t \left ( y^{(j)}, u, \xi^{(j)}; x \big. \right)
			+\sum_{j=1}^q \phi_t \left (z^{(j)}, u,  \eta^{(j)};x \big. \right)
			    -\phi_t \left (z'^{(j)}, u,\eta'^{(j)};x \big.\right),
		\end{align}
		where $\phi_t$ is given in \eqref{phase function}.\\ \ \\
\		
\emph{(ii)}\;
The \underline{amplitude $ c^{[\ell, q]}_t$} satisfies
$$ c^{[\ell, q]}_t(u,\tilde y, \tilde z,\tilde z',  \tilde \xi, \tilde \eta, \tilde \eta'\,; x, \h)
\sim_{\h \to 0^+} \sum_{j=0}^\infty c_{t,j}^{[\ell, q]}(u,\tilde y, \tilde z,\tilde z',  \tilde \xi, \tilde \eta, \tilde \eta'\,; x)\h^j$$ with
$c_{t,j}^{[\ell, q]}(u, \cdot ,\cdot \,; \cdot) \in
 C^\infty\left( B^k(\varepsilon), 
 C_0^\infty(V^{(\ell+2q)} \times \R^{n(\ell+2q)} \times U)
 \big.\right)$
  for $U,V \subset \R^n$ local coordinate charts as in \eqref{phase function}.
Moreover,
   \begin{align}\label{support c}
		&supp \,( c^{[\ell, q]}_t(u,\tilde y, \tilde z,\tilde z', \cdot \,; x, \h) ) \subset \notag \\
	     &\left\{ (\tilde \xi, \tilde \eta,\tilde \eta'): \;\;
		(x,\xi^{(i)}),(x,\eta^{(j)}),(x,{\eta'}^{(j)}) \in p_0^{-1}(E-c\varepsilon, E+c\varepsilon), \, x \in U, \;\; \;\;
		i\leq \ell,\; j\leq q \right\} \notag \\
	   &\subset \R^{n(\ell+2q)}.
		\end{align}\ \\
\
\noindent \emph{(iii)}\; The \underline{residual operator $K_x$} satisfies
		$$|\partial^\alpha_x \,\partial_ {(\tilde y, \tilde z, \tilde z')}^\beta\,
		 K_x(\tilde y, \tilde z, \tilde z', u)|= \mathcal O_{\alpha,\beta}(\h^\infty)$$
\indent \indent \indent	locally uniformly in $(\tilde y, \tilde z, \tilde z', u)  \in  V^{(\ell +2q)}  \times
		B^k(\varepsilon).$ \\ \ \\

\noindent {\bf Claim.}  For $\varepsilon>0$ and $|t| >0$ sufficiently small, there exists $C=C(t, \varepsilon, E, g_0)>0$ such that for
$ (\tilde \xi,   \tilde \eta, \tilde \eta') \in supp \,( c^{[\ell, q]}_t(u,\tilde y, \tilde z,\tilde z',\cdot \,; x, \h) )$, 
 \begin{equation}\label{nonzero derivative}
  \left| \partial_{u_\alpha} \Phi^{[\ell,q]}_t(\tilde y, \tilde z, \tilde z', u, \tilde \xi, \tilde \xi', \tilde \eta, \tilde \eta'  ;x) \right | \geq C> 0, 
 \end{equation} 
where this bound holds locally uniformly for  $(\tilde y, \tilde z, \tilde z') \in M^{(\ell +2q)}$, 
$u\in B^{k}(\varepsilon)$, and
 $x \in \mathcal W \cap (V^{-1}(E))^c$.\\

\noindent To prove this claim we first observe that since $\delta_{u_\alpha} g_u^{^{-1}}(x)= a(x) g_0^{^{-1}}(x)$ 
for $x \in \mathcal W$, 
\begin{align}\label{p_u approx}
\delta_{u_\alpha} p_u(x, \xi)  = a(x)|\xi|^2_{g_0(x)}.
\end{align}

Also, from the Taylor expansion of
the generating function \eqref{Taylor for S} around $s=0$,  together with \eqref{phase function}, we know that for  $x \in \mathcal W$ 
 \begin{equation}\label{approx. phase fxn}
 \phi_t(y,u,\eta;x)= \langle x-y ,\eta \rangle -t p_u(x,\eta) + \mathcal O(t^2),
 \end{equation} 
 where in (\ref{approx. phase fxn}), the error $\mathcal O(t^2)$ depends on $\tilde{p}.$ 
 Combining \eqref{Phi} with  \eqref{approx. phase fxn} and \eqref{p_u approx}, for  $x \in \mathcal W$ we get
\begin{align*}
	&\partial_{u_\alpha} \Phi^{[\ell,q]}_t(\tilde y, \tilde z, \tilde z', u, \tilde \xi, \tilde \eta, \tilde \eta'  ;x)= \nonumber \\
	%& =-t \left(\sum_{i=1}^\ell \left \langle  \delta_{ u_\alpha}g_u^{^{-1}}(x) \xi^{(i)} ,   \xi^{(i)} \right \rangle
	 %+  \sum_{j=1}^q \left(\left \langle  \delta_{ u_\alpha}g_u^{^{-1}}(x)  \eta^{(j)} , \eta^{(j)}  \right \rangle
	  % -\left \langle  \delta_{ u_\alpha}g_u^{^{-1}}(x) \eta'^{(j)} ,  \eta'^{(j)}  \right \rangle \right ) \right)  \nonumber \\
	% &\quad +\mathcal O (|x|^2) + \mathcal O(t^2) \nonumber \\
	& \qquad \quad=-t\,a(x)\left(\sum_{i=1}^\ell  | \xi^{(i)}|^2_{g_0(x)}
	  +  \sum_{j=1}^q | \eta^{(j)} |^2_{g_0(x)}
	   - \sum_{j=1}^q |\eta'^{(j)}|^2_{g_0(x)} + {\mathcal O}(|u|) \right)
	% +\mathcal O (|x|^2) 
	 + \mathcal O(t^2). \nonumber \\
%	 & =-t\,a(x)\, \ell\, (E  - V(x)) + \mathcal O (\varepsilon) +\mathcal O (|x|^2) + \mathcal O(t^2),
\end{align*} 
From the support conditions on the amplitude $c_t^{[l,q]}$ in \eqref{support c}, we have that
$ | \xi^{(j)} |^2_{g_0(x)} + V(x)= E +\mathcal O(\varepsilon)$, $ | \eta^{(j)} |^2_{g_0(x)} + V(x)= E +\mathcal O(\varepsilon)$
and $ | \eta'^{(j)} |^2_{g_0(x)} + V(x)= E +\mathcal O(\varepsilon)$ for $i\leq \ell$ and $j\leq q$. Therefore, for  $x \in \mathcal W$,
\begin{align} \label{ibp}
\partial_{u_\alpha} \Phi^{[\ell,q]}_t(\tilde y, \tilde z, \tilde z', u, \tilde \xi, \tilde \eta, \tilde \eta'  ;x)=-t\,a(x) \, \, \Big( \ell\, (E  - V(x)) + (\ell+2q+1) \mathcal O (\varepsilon) \Big) 
+ \mathcal O(t^2),
\end{align}
and so 
\begin{equation} \label{ibp2}
\left| \partial_{u_\alpha} \Phi^{[\ell,q]}_t(\tilde y, \tilde z, \tilde z', u, \tilde \xi, \tilde \xi', \tilde \eta, \tilde \eta'  ;x) \right |
\geq |t a(x)|  \cdot | \ell\, (E  - V(x)) + 2 \tilde p \mathcal O (\varepsilon) |
+ \mathcal O(t^2), \end{equation}
uniformly in all variables. 
Given a compact subset $K \subset \mathcal W \cap (V^{-1}(E))^c,$  for $x \in K,$ one has $|V(x) - E| \geq \frac{1}{C_0} >0$ for some constant $C_0>0.$ Since $a(x)\neq 0$  for all $x \in \mathcal W,$ from (\ref{ibp2}) it follows that for such a given compact  set $K$ and  number of odd moments $\tilde{p} \in {\mathbb Z}^+,$ we can choose $\epsilon >0$ and $t \neq 0$ sufficiently small (depending on $\tilde{p}$ and $K$) so that the RHS of (\ref{ibp2}) is uniformly bounded away from zero. 
 We conclude that the claim in \eqref{nonzero derivative} holds for  $x \in \mathcal W \cap (V^{-1}(E))^c$. \\

We then use the operator
$\Big( \frac{\h}{i\, \partial_{u_\alpha} \Phi^{[\ell,q]}_t} \Big) \,\frac{\partial }{\partial u_\alpha},$
 to repeatedly  integrate by parts in \eqref{odd moments computation} and obtain
\
$$B_{t,u}^{[\ell, q]}(\h) (\tilde y, \tilde z, \tilde z';x,h) = \mathcal O (\h^{\infty}) $$
locally uniformly for  $(\tilde y, \tilde z, \tilde z') \in  M^{(\ell +2q)}$, 
$u \in B^{k}(\varepsilon)$ and
 $x \in \mathcal W\cap (V^{-1}(E))^c$.
We note that there are no boundary term contributions arising from the integration by parts since
 $\chi(u)=0$ for $u \in \partial B^k(\varepsilon)$.
\
\noindent From \eqref{odd moments computation}  it follows that
\
$$\int_{B^k(\varepsilon)} \left(\varphi_{h,t}^{(u)}(x)\right)^\ell
\	   \left|\varphi_{h,t}^{(u)}(x)\right|^{2q} \;d \nu(u)= \mathcal O (\h^{\infty}),$$
locally uniformly in $x \in \mathcal W \cap (V^{-1}(E))^c$.\\

\end{proof}

  \section{Variance}\label{section: variance}

As explained in the Introduction (see (\ref{eq:variance})), provided the perturbation is admissible at $x \in (V^{-1}(E))^c$,
the  case $p=1$ in Proposition \ref{negligible odd moments}  shows that our random variables are semiclassically centered in the sense
$$\mathbb E \left[ \Re(\varphi_{h,t}^{(\cdot)}(x)) \right]= \mathcal O(\h^\infty).$$
Therefore,
$$Var \left[ \Re(\varphi_{h,t}^{(\cdot)}(x)) \right] =
\int_{B^k(\varepsilon)} | \varphi_{h,t}^{(u)}(x)|^2\,d \nu(u) +\mathcal O(\h^\infty).$$
It then follows that studying the variance is equivalent to understanding the behavior of the right hand side in the
previous equality.  We compute the asymptotics of the RHS in the next Proposition.
\begin{proposition}\label{A_x}
Let $g_u$ be admissible at $x_* \in M$  and let $\mathcal U$ be the neighborhood of $x_*$ given in Remark \ref{neighborhood U}.
  For $\varepsilon>0$ and $|t|>0$  sufficiently small,
there exist a choice of coordinates  $u'' \in B^{k-n}(\varepsilon)$
and corresponding  operators $A_{t,x,u''}(\h) \in \Psi_{cl}^{0,-\infty}(M)$
defined for all $(x,u'') \in \mathcal U \times B^{k-n}(\varepsilon),$ such that %for $\varphi_{\h}^{(u)} = e^{-itP_{u}(\h)/\h} \varphi_{\h},$
\begin{equation}\label{eq:variance2}
\int_{B^{k}(\varepsilon)} \left|\varphi_{h,t}^{(u)}(x)\right|^{2} d \nu(u)
= c_k(\varepsilon) \int_{B^{k-n}(\varepsilon)}
	\left \langle  A_{t,x,u''}(\h) \varphi_\hbar,
	\varphi_\hbar  \Big. \right \rangle_{L^2(M)}  du''+ \mathcal O(\h^\infty).
\end{equation}
%In addition,  there exists a constant $C_1=C_1(\varepsilon, t, E, g_0)>0$ such  that
%\begin{equation}\label{bound symbol}
%	|\sigma_0\left(A_{t,x,u''}(\h) \big.\right) (y,\eta) | > \frac{C_1}{2}>0
%\end{equation}
%uniformly for $(y,\eta) \in p_0^{-1}(E)$ and $(x,u'') \in B(x_*, \tau_\varepsilon) \times B^{k-n}(\varepsilon)$.\\
\end{proposition}\  \smallskip

\begin{proof}
By Remark \ref{neighborhood U}, we choose $\mathcal U\subset M$ to be an open neighborhood of $x_*$ so that the admissibility condition \eqref{cond 1} holds
on $\mathcal U$ . That is,  given the constant  $c>0$  in (\ref{constant})  and some subset of $n$ coordinates of $u$, which we denote  $u' \in B^n(\varepsilon)$, so that 
 the matrix 

\begin{equation}\label{admi.}
d_{u'} d_\xi ( p_u(x,\xi)) \quad \text{ is invertible for }\quad
\;  (x,\xi) \in p_0^{-1}(E-c\varepsilon, E+c\varepsilon), \, (x,u) \in \mathcal U \times  B^k(\varepsilon).
\end{equation} 
 We write $u''\in B^{k-n}(\varepsilon)$ for the omitted variables
 and assume that the coordinates of $u$ are ordered so that $u=(u',u'')$.\\

Write $W_{t,u}(\h)(x,y)$ for the Schwartz kernel of $W_{t,u}(\h).$ Then,
 for $u''\in B^{k-n}(\varepsilon)$ and  $x \in \mathcal U,$ we define a new family of operators
\begin{equation}\label{W_{x,E}}
	\hat W_{t,x,u''}(\h):C^{\infty}(M) \to C^\infty_0(B^n(\varepsilon)),
\end{equation}
with Schwartz kernels 
$$\hat W_{t,x,u''}(\h)(u',y):= \chi(u) \cdot W_{t,u}(\h)(x,y), \quad  \;\; u=(u',u'') \in B^k(\varepsilon),$$ 
where we continue to write $\chi$ for the cut-off function appearing in the definition of the
probability measure $\nu$ in (\ref{normalizing factor}).
 By \eqref{approx of varphi},
 $$\chi(u) \, \varphi_{h,t}^{(u)}(x) = \chi(u) \: [W_{t,u}(\h)\varphi_\h](x) + \mathcal O(\h^\infty)=[\hat W_{t,x,u''}(\h)\varphi_\h](u')+  \mathcal O(\h^\infty),$$
and so,
\begin{align}\label{key computation}
&\int_{B^{k}(\varepsilon)} \left|\varphi_{h,t}^{(u)}(x)\right|^{2} d \nu(u)
=   c_k(\varepsilon) \int_{B^{k}(\varepsilon)} \left| \hat W_{t,x,u''}(\h) \varphi_\hbar (u') \right|^2 du  + \mathcal O(\h^\infty) \notag \\
& \quad \qquad = c_k(\varepsilon)	\int_{B^{k-n}(\varepsilon)}
	\left \langle  \hat W_{t,x,u''}(\h) \varphi_\hbar,
	 \hat W_{t,x,u''}(\h) \varphi_\hbar \Big. \right \rangle_{L^2(B^{n} (\varepsilon))}  du''   + \mathcal O(\h^\infty).  \notag\\
\end{align}

From \eqref{kernel of W_{t,u}}, the Schwartz kernel of $ \hat W_{t,x,u''}(\h)$ is given by
	\begin{equation}\label{kernel of W_{x,E}}
	\hat W_{t,x,u''}(\h)(u',y)= \frac{1}{(2\pi h)^n} \int_{\R^n} e^{\frac{i}{\h} \phi_t(y,u',u'', \xi; x) } a_t(u,y,\xi;x,\h)  \chi(u)\; d\xi + K_x(y,u),
	\end{equation}
where $|\partial^\alpha_x \partial^\beta_y K_x(y,u)|= \mathcal O_{\alpha,\beta} (\h^\infty)$ uniformly in
$(x,y,u) \in U \times V \times B^k(\varepsilon)$ for $\varepsilon>0$ small, where $U,V \subset \R^n$ are
local coordinate charts with $U \subset \mathcal U.$ 
The amplitude $ a_t(u,y,\xi;x,\h) \sim \sum_{j=0}^\infty a_j(u,y,\xi;x)\h^j$ with
$a_j(u, \cdot ,\cdot \,; \cdot) \in
 C^\infty( B^k(\varepsilon), C_0^\infty(V \times \R^n \times U)).$ Moreover, we recall from (\ref{amplitudes}) that $\text{supp} \;(a_t(u, \cdot, \cdot \,; \cdot , \h))  \subset
 \{(y,\xi,x) \in T^*U \times V: (x,\xi)\in p_0^{-1}(E-c\varepsilon, E+c\varepsilon), \, y = d_{\xi}S(t,u,\xi;x) = x + {\mathcal O}_{u}(t) \}.$\\

By the same argument presented in \cite[Prop. 4.1]{ET}, it can be shown that
for $x \in U,$ $\varepsilon>0$ and $|t| \neq 0$ small enough,
$ \hat W_{t,x,u''}(\h) \in I^{0,-\infty}_{cl} (M \times B^n(\varepsilon); \Gamma_{t,x,u''})$
with 
\begin{align}\label{Lagrangian}
	 \Gamma_{t,x,u''}&:= \{(u',d_{u'}S(t, u, \eta; x)), d_\eta S(t, u, \eta; x), \eta): \;
	 (d_\eta S(t, u, \eta; x),\eta) \in supp \,\chi_E^{(0)} \} \notag \\
	 &\hspace{8cm} \subset T^*B^n(\varepsilon) \times T^*M.
	 \end{align}
where $u:=(u',u'')\in B^k(\varepsilon)$ for $\varepsilon$ small, and $x\in U$.  It remains to show that $\Gamma_{t,x,u''}$ is a canonical graph.\\
We recall   that if  $(d_\eta S(t,u,\eta;x),\eta) \in supp\, \chi_E^{(0)}$ for $u \in B^k(\varepsilon)$ then one gets that 
$(x,d_xS(t,u,\eta;x)) \in \text{supp} \chi_E^{(u)} \subset  p_0^{-1}((E-(c+1)\varepsilon, E+(c+1) \varepsilon))$ with $c>0$ as in \eqref{constant}. 
 %This is a consequence of the fact that  according to \eqref{Taylor for S},
%$d_\eta S(t,u,\eta;x)=x-t \partial_{\eta}p_u(x,\eta) + \mathcal O(t^2)$. 
By \eqref{Lagrangian} and the admissibility assumption (B) it follows that by possibly shrinking $t\neq0$, we can ensure that  there is a constant $C_0>0,$ such that  the
non-degeneracy condition
 \begin{equation}\label{non-degeneracy}
  \det(d_{u'}d_\eta \phi_t (y,u',u'',\eta;x)) =|t|^n \left(\det(d_{u'}d_\eta p_u(x,\eta)) + \mathcal O (t ^2)\right) \geq C_0 |t|^n,
  \end{equation}
 holds uniformly for $(u,y,\eta)$ with $(y,\eta) \in \text{supp} \chi_E^{(0}$ and $y=d_\eta S(t,u,\eta;x).$
Now, for $u''$ fixed and $x \in U$, consider the map
$$(u',y, \eta) \mapsto d_\eta \phi_t (y,u',u'',\eta;x), \quad \quad  (u', \tau; y,\eta) \in \Gamma_{t,x,u''} .$$

\noindent We claim that due to the the non-degeneracy condition  \eqref{non-degeneracy},
the Lagrangian \eqref{Lagrangian} is a canonical graph. Indeed,  \eqref{non-degeneracy} allows us to apply the
Implicit Function Theorem and locally write $u'=u'(y,\eta)$ satisfying
\begin{equation}\label{u'}
u'=u'(y,\eta) \quad \text{when}\quad d_\eta \phi_t (y,u',u'',\eta;x)=0,
\end{equation}
for $x \in U$.
\
Then,  taking into account that for $x \in U$
$$d_\eta \phi_t (y,u',u'',\eta;x)=0 \quad \text{when}\quad  y=d_\eta S(t,u',u'',\eta;x),$$
we write  $(y,\eta)\in V \times \R^n$ as  local parametrizing variables for $\Gamma_{t,x,u''}$ as in  \eqref{Lagrangian}  and  get: 
\begin{align} \label{u'=u'(y,xi)}
\Gamma_{t,x,u''}=& \left\{\left (u' (y,\eta), d_{u'} S \left(t,u'(y,\eta),u'',\eta; x \big. \right)\;;\; y, \eta \Big.\right): \right.\\
&\hspace{4cm}\left. (y,\eta)\in supp \,\chi_E^{(0)}, \,\, y=d_\eta S(t,u',u'',\eta;x)  \Big.\right\}.
\end{align} 
	
For $u'' \in B^{k-n}(\varepsilon)$ and $x \in U$ define the operators
$$A_{t,x,u''}(\h): C^\infty(M) \to C^\infty(M),$$
\begin{equation}\label{A_{x,E}}
A_{t,x,u''}(\h):=\left( \hat W_{t,x,u''}(\h) \Big.\right)^* \circ (\hat W_{t,x,u''}(\h) ).
\end{equation}

Since $ \hat W_{t,x,u''}(\h) \in I^{0,-\infty}_{cl} (M \times B^n(\varepsilon); \Gamma_{t,x,u''})$
and the immersed Lagrangian $\Gamma_{t,x,u''}$ is a canonical graph,
the operator
		$$A_{t,x,u''}(\h) \in \Psi_{cl}^{0,-\infty}(M),$$
for $x \in U$ and $u'' \in B^{k-n}(\varepsilon)$.
From \eqref{key computation} and \eqref{A_{x,E}} it follows that
$$\int_{B^{k}(\varepsilon)} \left|\varphi_{h,t}^{(u)}(x)\right|^{2} d \nu(u)
=c_k(\varepsilon)	 \int_{B^{k-n}(\varepsilon)}
	\left \langle  A_{t,x,u''}(\h) \varphi_\hbar,
	\varphi_\hbar  \Big. \right \rangle_{L^2(M)}  du''   + \mathcal O(\h^\infty).$$

\end{proof}

%------------------------------------------------------------------------------------------------------------

%-------------------------------------
\subsection{Proof of Theorem \ref{odd moments and variance}}
\
Since $M$ is compact we choose a finite covering
$$M \subset \bigcup_{j=1}^N \mathcal V_{x_j}$$
where $x_j \in M$ and $\mathcal V_{x_j}=\mathcal W_{x_j} \cap \mathcal U_{x_j}$.
Here $\mathcal W_{x_j}$ is given by part \eqref{cond 2} of the admissibility condition at $x_j$,
and $\mathcal U_{x_j}$ is given in  Remark \ref{neighborhood U}.\\

Fix $j \in \{1, \dots, N\}$ and let $x\in \mathcal V_{x_j}$.
To prove the first part of Theorem \ref{odd moments and variance} we note that
Proposition \ref{A_x} gives $A_{t,x,u''}(\h) \in \Psi_{cl}^{0,-\infty}(M).$ Thus,   by $L^2$ boundedness there
exists a constant $C_j=C_j(\varepsilon, t, E, g_0)>0$ such that
$$\langle  A_{t,x,u''}(\h) \varphi_\hbar, \varphi_\hbar  \rangle_{L^2(M)} \leq C_j$$
uniformly in $(x,u'', \h) \in \mathcal V_{x_j}\times B^{k-n}(\varepsilon) \times (0,\h_0]$.
%We obtain a lower bound from \eqref{bound symbol} and the weak Garding inequality,
%$$\langle  A_{t,x,u''}(\h) (\varphi_\hbar), \varphi_\hbar \rangle_{L^2(M)} \geq C_1^j >0 $$
%uniformly in $(x,u'', \h) \in B_j(x_*, \tau_\varepsilon)\times B^{k-n}(\varepsilon) \times (0,\h_0]$.
  Therefore, from \eqref{eq:variance} and \eqref{eq:variance2} one can choose a positive constant $C >0$
   so that the first part of the statement of
   Theorem \ref{odd moments and variance} holds uniformly for $x \in K,$ where $K  \subset (V^{-1}(E))^c$ is any compact subset.\\

To prove the second part of Theorem \ref{odd moments and variance}
 regarding the odd moments we simply apply Proposition \ref{negligible odd moments} in each neighborhood $\mathcal V_{x_j}$.
{\flushright \qed}

%----------------------------------

 \subsection{Proof of Theorem \ref{moments function}}
 From Proposition \ref{A_x}  and   equation \eqref{eq:variance},
\begin{align}\label{eq 1}
\lim_{\h\to 0^+} Var \left[ \Re \left(\varphi_{h,t}^{(\cdot)}(x)\right)\right]
&=\lim_{\h\to 0^+}
 \int_{B^{k}(\varepsilon)} \left|\varphi_{h,t}^{(u)}(x)\right|^{2} d\nu(u) \notag \\
& =\lim_{\h\to 0^+}
	c_k(\varepsilon) \int_{B^{k-n}(\varepsilon)}
	\left \langle  A_{t,x,u''}(\h) \varphi_\hbar, \varphi_\hbar  \Big. \right \rangle_{L^2(M)}  du''. \notag \\
\end{align}

Since $(\varphi_\h)$ is a quantum ergodic sequence,
\begin{equation}\label{eq 2}
\lim_{\h \to 0^+}\left \langle  A_{t,x,u''}(\h) \varphi_\hbar,
	\varphi_\hbar  \Big. \right \rangle_{L^2(M)}
	=\frac{1}{|p_0^{-1}(E)|} \int_{p_0^{-1}(E)} \sigma_0( A_{t,x,u''}(\h)) (y,\eta) \,d\omega_E(y,\eta).
\end{equation}

In addition, following the same argument presented in Corollary 4.2 of \cite{ET}, the 
principal symbol can be locally written as
		\begin{align}\label{ppal symbol}
		\sigma_0\left(A_{t,x,u''}(\h) \big.\right)  (y,\eta)
		&=|\chi^{(u',u'')}_E(x,\eta)|^2 \frac{|\det(d_x \pi G_{(u',u'')}^{-t}(x,\eta))|}{|\det(d_{u'} d_\eta S(t, x, \eta; u',u''))|}
		\; \chi^2(u',u'') \notag \\
		&=|\chi^{(u',u'')}_E(x,\eta)|^2 \frac{|\det(d_x \pi G_{(u',u'')}^{-t}(x,\eta))| }{\,| t|^n\, |\det(d_{u'} d_\eta \,p_{(u',u'')}( x, \eta))|}\;
		 (1+ {\mathcal O}(t)) \, \chi^2(u',u'') \notag \\
		\end{align}		
for  $u'=u'(y,\eta)$ parametrizing the Lagrangian $\Gamma_{t,x,u''}$ regarded as a canonical graph.
%In particular, \eqref{bound symbol} holds.
The first statement of Theorem \ref{moments function} then follows by combining \eqref{eq 1}, \eqref{eq 2} and the expression for the
principal symbol \eqref{ppal symbol}.\\

The second statement of Theorem \ref{moments function} about odd moments is a direct application of  the second part of Theorem \ref{odd moments and variance}.
{\flushright \qed}

%------------------------------------------------------------------------------------------------------

%------------------------------------------------------------------------------------------------------

\section{Admissible perturbations}\label{section: admissible perturbations}

In this section we study the geometry behind the admissibility condition and show that
perturbations satisfying such conditions always exist. It is clear that one can always
have perturbations satisfying part \eqref{cond 2} of the admissibility condition.
We therefore focus on proving the existence of metric perturbations satisfying condition \eqref{cond 1}.
The symbol $p_{u}:T^*M \to T^*M$ defined in \eqref{p_u} has the form
$$p_{u}(x,\xi)=\sum_{i,j=1}^n g_u^{ij}(x)\xi_i \xi_j +V(x) .$$
%In geometric terms, a perturbation $g_u$ with $u \in B^k(\varepsilon)$
%satisfies part \eqref{cond 1} of the admissibility condition provided the map
%$$\Omega_\xi : B^k(\epsilon) \to \R^n, \quad \quad \Omega_\xi(u):= g_u^{-1}(\xi)= \left( \sum_{l=1}^n g_u^{il}(x)\xi_l\right)_{i=1,\dots,n}$$
%is a submersion at $u=0$ for all  $(x,\xi) \in p_0^{-1}(E-c\varepsilon, E+c\varepsilon)$ and some $c>0$. \\

Write $\mathcal M$ for the space of Riemannian metrics on $M$.  For each coordinate $u_s$ of $u$ define the symmetric tensor
 $h_{u_s}:=\delta_{u_s}g_u^{-1}$ and write in local coordinates
	\begin{equation}\label{def: h}
	h_{u_s}= h_{u_s}^{ij}\, dx_i \otimes dx_j, \quad \quad h_{u_s}^{ij}:= \delta_{u_s}g_u^{ij}.
	\end{equation}
It is straight forward to check that
$\partial_{u_s} \partial_{\xi_i} p_u(x,\xi) \big|_{u=0}=2\, \sum_{l=1}^n h_{u_s}^{l i}(x)\,\xi_l.$
Thereby, a metric perturbation satisfies condition \eqref{cond 1}  provided  there exist $c>0$ and
 an $n$-tuple $u'=(u_1, \dots, u_n)$ of coordinates of $u$
		so that  for all  $(x,\xi) \in p_0^{-1}(E-c\varepsilon, E+c\varepsilon)$,  the matrix
		$\left (\sum_{l=1}^n h_{u_j}^{l i}(x)\,\xi_l\right)_{i,j=1,\dots,n}$
is invertible.
By definition, the notion of admissibility depends on the direction, inside the space of symmetric tensors,
in which $g_0$ is deformed.
In what follows we show that the admissibility condition is directly related to performing the deformation $g_u$
 in sufficiently many {\em volume preserving} directions, described below.

Let $\mathcal P$ denote the multiplicative group of positive smooth functions on $M$, which we refer to
as \emph{pointwise conformal deformations}. $\mathcal P$  acts on $\mathcal M$ by multiplication
$$ \mathcal P \times \mathcal M \to \mathcal M, \quad \quad (p,g) \to pg.$$
Given  $g_0 \in \mathcal M$,  the orbit of $g_0$ under $\mathcal P$ denoted by $\mathcal P \cdot g_0$,
is a closed submanifold of $\mathcal M$ with tangent space at $g_0$ given by
\begin{equation}\label{conf def}
T_{g_0} ( \mathcal P \cdot g_0)= \{ v \in S^2(M) : \;\; v=f\,g_0,\; f \in C^\infty(M,\mathbb R) \}.
\end{equation}

Let $\mu$ be a volume form on $M$ and define $\mathcal N_{\mu}:= \{ g \in \mathcal M : \mu=\mu_g\}$ 
where $\mu_g$ denotes the Riemannian volume measure associated to $g$.  Pointwise conformal 
transformations $g_0 \mapsto f g_0$ multiply the volume form $\mu_{g_0}(x)$ at a point $x$ by $(f(x))^{n/2}$.  Transverse 
to the orbit of $g_0$ by the action of pointwise conformal transformations is the sub manifold 
${\mathcal N}_{\mu_{g_0}}$ of all metrics $g$ on $M$ with the  fixed volume form;
equivalently, the determinant $\det(g_{ij}(x))$ is preserved for all $x$.  It is well-known 
that the tangent space to the space of symmetric matrices with fixed determinant consists of symmetric 
 traceless matrices.  Accordingly, it can be shown (cf. \cite{Ebi}) that the tangent space 
$T_{g_0}{\mathcal N}_{\mu_{g_0}}$ is given by
	\begin{equation}\label{vol pres}
	T_{g_0} ( \mathcal N _{\mu_{g_0}} )= \{ v \in S^2(M) : \;\; (tr_{g_0} v)(x)=0,\forall x\in M\}.
	\end{equation}
For every metric $g_0 \in \mathcal M$ the space of symmetric tensors has the pointwise orthogonal splitting
$$T_{g_0} \mathcal M =T_{g_0} ( \mathcal N _{\mu_{g_0}} ) \oplus T_{g_0} ( \mathcal P \cdot g_0)$$
where every $v \in S^2(M)$ is decomposed as $v= (v- \frac{tr_{g_0}v}{n} g_0) + \frac{1}{n} (tr_{g_0}v)\, g_0$.\\

Let $g_u$ be a metric deformation of $g_0$; we continue to write $h_{u_s}= \delta_{u_s} g_u^{-1}$. 
Working in geodesic normal coordinates at $x_*$, it is not difficult to show that
volume-preserving deformations are characterized by the condition $tr_{g_0^{-1}}(h_{u_s} (x))=0$ 
for all $s$.
We shall show below that the admissibility condition holds for such deformations.

\subsection{Surfaces}
On surfaces, we claim that perturbations $g_u$  that have two linearly independent
$u$-derivatives in the volume preserving directions are admissible.
\begin{proposition}
Let $(M,g_0)$ be a compact Riemannian surface. Let $E$ be a regular value of $p_0$.
Suppose $g_u$ with $u \in B^k(\varepsilon)$  is a  perturbation of  $g_0$ such that  there exist  
two coordinates $u'=(u_1, u_2)$ of $u$ for which
$h_{u_1}(x)$ and  $h_{u_2}(x)$ (as defined in \eqref{def: h}) are linearly independent tensors with 
$tr_{g_0^{-1}}(h_{u_1})(x)=tr_{g_0^{-1}}(h_{u_2})(x)=0$
at some $x \notin V^{-1}(E)$. 
Then, for $\varepsilon$ small enough, the perturbation $g_u$ satisfies part  \eqref{cond 1} of the admissibility condition at $x$.
 \end{proposition}

\begin{proof}
Let $x_*\in M$ be
such that $x$ belongs to a geodesic ball cantered at $x_*$, and consider
normal coordinates at $x_*$. In these coordinates,
$g_{0_{ij}}(x)=\delta_{ij}+\mathcal O (|x|^2)$ for $x$ being at a small distance $|x|$
from $x_*$. Therefore, since  $tr_{g_0^{-1}}(h_{u_s})(x)=0$ for $s=1,2$, we have $h_{u_s}^{11}(x)=-h_{u_s}^{22}(x)+ \mathcal O (|x|^2)$ for $s=1,2$.
It is straight forward to check

$$\det \left(\sum_{j=1}^2 h_{u_s}^{ij}(x)\,\xi_j \right)_{s,i=1,2}
	= |\xi|_{g_0(x)}^2 \left(  \det  \begin{pmatrix}
						h_{u_1}^{11}(x)  & h_{u_2}^{11}(x) \\
						h_{u_1}^{12}(x)  & h_{u_2}^{12}(x)
						\end{pmatrix}
	+  \mathcal O (|x|^2) \right).$$
	
Since we are only interested in what happens when  $|\xi|^2_{g_0(x)}+V(x)=E+\mathcal O(\varepsilon)$,
the result follows from the assumption
$V(x)\neq E$ and the fact that $h_{u_1}(x)$ and  $h_{u_2}(x)$ are linearly independent tensors.
\end{proof}

%-----------------------------------------------------------------------------------
\subsection{Manifolds}

In what follows we show that on an $n$-dimensional manifold we can always have
admissible perturbations.\\

Let $M$ be an $n$-dimensional compact manifold and fix $x_* \in M$.
Consider a geodesic normal coordinate system at $x_*$.
 We shall consider deformations of the reference metric $g_0$ that, as in the surface case,
preserve the volume form. Infinitesimally, as explained in \eqref{vol pres}, the
corresponding quadratic
form is given by a traceless symmetric matrix. The space of traceless symmetric tensors at $x \in M)$
has dimension
\begin{equation}\label{dimension of traceless}
\kappa_n:= \frac{n^2+n-2}{2},
\end{equation}
and the basis of the space of such forms is given by
$$
\xi_1^2-\xi_i^2 , \;\; 2\leq i\leq n; \quad \text{and} \quad  \xi_j\xi_k, \;\;
1\leq j<k\leq n,
$$
for $\xi=(\xi_1, \dots, \xi_n) \in T^*_xM$; we denote these polynomials evaluated at $x=x_*$ by 
$q_j(\xi)$ for $j=1, \dots, \kappa_n$.\\

Note that since we are using normal coordinates centred at $x_*$ then $|\xi|^2_{g_0(x_*)}= \sum_i \xi_i^2$,  
and remark that the polynomials $q_j(\xi)$  form a basis in  
the space of spherical harmonics of degree two on the sphere $S^{n-1}_{x_*}=\{\xi \in T^*_{x_*}M:  |\xi|_{g_0(x_*)}=1\}$ with the round metric $g_{S^{n-1}_{x_*}}$.
In the proof we shall use a computation showing that the round metric on $S^{n-1}_{x_*}$ is an extremal 
for the second eigenvalue of the Laplacian in the space of nearby Riemannian metrics.  We refer 
to \cite{EI,Nad,Tak} and references therein for a description of general theory of such metrics.  

Below we summarize several well-known facts about extremal metrics.  Let $g_0$ be an extremal metric on a 
compact $d$-dimensional manifold $N$.  Then 
\begin{itemize}
\item[a)] If $(N,g_0)$ is a homogeneous space (e.g. a round $S^{n-1}$), then the metric $g_0$ 
is extremal for {\em all} eigenvalues of the Laplacian $\Delta_{g_0}$.  
\item[b)]  If $g_0$ is extremal for an eigenvalue $\lambda$ (of multiplicity $m$), and 
$\varphi_1,\ldots,\varphi_m$ form an orthonormal basis of the corresponding eigenspace $E_\lambda$, 
then 
\begin{equation}\label{critical:metric}
\sum_{j=1}^m d\varphi_j\otimes d\varphi_j=(\lambda/2)g_0.  
\end{equation}
\end{itemize}
We remark that one can show that $m>d$ in \eqref{critical:metric}.

\begin{claim}\label{full:rank}
Assume that $\{\varphi_1,\ldots,\varphi_m\}$ satisfy \eqref{critical:metric}.  Then the (jacobian) matrix 
$$\frac{\partial(\varphi_1,\ldots,\varphi_m)}{\partial(x_1,\ldots,x_d)}$$ has the maximal possible rank $d$.  
\end{claim}
\begin{proof}
Assume that the rank is less than $d$.  Than the quadratic form $\sum_{j=1}^m d\varphi_j\otimes d\varphi_j$ 
cannot be positive-definite; however by assumption it is proportional to the positive-definite Riemannian  
metric $g_0$ on $N$.  Contradiction finishes the proof. 
\end{proof}
The existence of admissible perturbations will follow from Claim \ref{full:rank}.  

We shall show that a metric perturbation is admissible if it is of the form 
\begin{equation}\label{admissible:taylor}
g_u^{-1}(x)=g_0^{-1}(x) + \sum_{j=1}^{\kappa_n} u_j\, h_j(x),
\end{equation}
where $h_j(x)(\xi,\xi)$ are homogeneous of degree $2$ in the $\xi$ variables and 
are required to satisfy the following conditions: 
\begin{itemize}
\item[(a)] $h_j(x_*)(\xi,\xi)=q_j(\xi),\;1\leq j\leq\kappa_n$; 
\item[(b)] $h_j$ are $C^2$ tensors.  
\end{itemize}

%For $j \in 1, \dots, \kappa_n$  define the symmetric 2-tensor $h_j$ on $B(x_*, \text{inj}M)$
%by  $h_j(x)(\xi, \xi)= q_j(x,\xi)$ for all $x \in B(x_*, \text{inj} M)$ and $\xi \in T_x^*B(x_*, \text{inj}M)$.

\begin{proposition}\label{local:sufficient}
Let a perturbation of  $g_u$ of $g_0$ have the form \eqref{admissible:taylor} with $h_j$ satisfying conditions 
(a) and (b).  Then there exist $\varepsilon>0$ and $\delta>0$ such that 
for all $u \in B^{n}(\varepsilon)$ the perturbation $g_u$ 
satisfies part \eqref{cond 1} of the admissibility condition at $x$ for $x\in B(x_*,\delta)$.
\end{proposition}

\begin{proof}
   Clearly,
$\partial _{u_j} (|\xi|_{g_u(x_*)})=q_j(\xi)$,
and hence the $j$-th column (say) of the mixed hessian matrix $d_\xi d_u p_u(x_*,\xi) $
corresponds to the
gradient $d_\xi q_j(\xi)$.

Now, since the sphere $S^{n-1}_{x_*}$ is a homogeneous space, the round metric $g_{S^{n-1}_{x_*}}$
is a critical metric for the corresponding eigenvalue functional
$g \mapsto \lambda(g)\cdot {\rm Vol}(g)^{2/n}$, where
$\lambda$ denotes the second positive eigenvalue (without multiplicity)
of the Laplacian.

By \eqref{critical:metric} (\cite{EI,Nad,Tak}), the $L^2$-normalized basis of the eigenspace $E(\lambda)$
(which can be chosen as $\{q_1(\xi),\ldots,q_{\kappa_n}(\xi)\}$ in our case)
satisfies
$$ 
\sum_{j=1}^{\kappa_n} d_\xi q_j \otimes d_\xi q_j=c \lambda\,  
g_{S^{n-1}_{x_*}},\qquad c\neq 0.
$$
By Claim \ref{full:rank}, the subspace spanned by $d_\xi q_1(\xi),\ldots ,
d_\xi q_{\kappa_n}(\xi)$ has the full dimension $n-1$ in $T_\xi^*(S^{n-1}_{x_*})$ at any point $\xi \in S^{n-1}_{x_*}$.  
This shows that
$\{d_\xi q_j(\xi):\;1\leq j\leq \kappa_n\}$ span the full $T_\xi^*(S_{x_*}^{n-1})$, which proves the required
non-degeneracy condition.

Next, since $h_j$ are $C^2$ in $(x,\xi)$ by condition (b), the rank of   $d_\xi d_u p_u(x,\xi)$ changes continuously in $x$ and 
so is equal to $n-1$ for $x\in B(x_*,\delta)$ on $M$ for some $\delta>0$.   
\end{proof}

%---------------------------------------------------------------------------------------------------------------------------------------------

\end{document}